\documentclass[12pt]{amsart}
\pdfoutput=1
\usepackage{amssymb}
\usepackage{amsmath}
\usepackage{amsfonts}
\usepackage[usenames]{color}
\usepackage{graphicx}
\usepackage{array}
\usepackage{psfrag}
\usepackage{color}
\usepackage{ulem}
\usepackage{fancyhdr}
\usepackage[table,xcdraw]{xcolor}
\usepackage{multirow}
\usepackage[pagewise]{lineno}

\makeatletter
 
 \@addtoreset{equation}{section}
\makeatother

\newtheorem{theorem}{Theorem}
\newtheorem{lemma}[theorem]{Lemma}
\newtheorem{proposition}[theorem]{Proposition}

\theoremstyle{definition}

\theoremstyle{remark}
\newtheorem{remark}[theorem]{Remark}

\numberwithin{equation}{section}

\newcommand{\R}{\mathbb{R}}

\newcommand{\dfn}[1]{\textit{#1}}

\newcommand{\G}{\Gamma}

\begin{document}

\title{New bounds on maximal linkless graphs}

\date{\today}
\author{
Ramin Naimi, Andrei Pavelescu, and Elena Pavelescu
}

\address{
}

\maketitle
\rhead{IK}

\begin{abstract}
We construct a family of maximal linklessly embeddable graphs on $n$ vertices and $3n-5$ edges for all $n\ge 10$, and another family on $n$ vertices and $m< \frac{25n}{12}-\frac{1}{4}$ edges for  all $n\ge 13$.
The latter significantly improves the lowest edge-to-vertex ratio for any previously known infinite family.
We construct a family of graphs showing that the class of maximal linklessly embeddable graphs differs from the class of graphs that are maximal without a $K_6$ minor studied by L.\ J{\o}rgensen.
We give necessary and sufficient conditions for when the clique sum of two maximal linklessly embeddable graphs over $K_2$, $K_3$, or $K_4$ is a maximal linklessly embeddable graph, and use these results to prove our constructions yield maximal linklessly embeddable graphs.

\end{abstract}
\vspace{0.1in}

\section{Introduction}

All graphs in this paper are finite and simple.
A graph is \dfn{intrinsically linked} (IL) if every embedding of it
in $\R^3$ (or, equivalently, $S^3$) contains a nontrivial 2-component link.
A graph is \dfn{linklessly embeddable} if it is not intrinsically linked (nIL).
A nIL graph  $G$ is \textit{maxnil} if it is not a proper subgraph of a nIL graph of the same order.
The combined work of Conway and Gordon \cite{CG},
Sachs \cite{Sa} and Robertson, Seymour and Thomas \cite{RST} fully characterized IL graphs: a graph is IL if and only if it contains a graph in the Petersen family as a minor.
The Petersen family consists of seven graphs obtained from $K_6$ by $\nabla Y-$moves and $Y\nabla-$moves, as described in Figure~\ref{fig-ty}. 
The $\nabla Y-$move  and the $Y\nabla-$move preserve the IL property.


\begin{figure}[htpb!]
\begin{center}
\begin{picture}(160, 50)
\put(0,0){\includegraphics[width=2.4in]{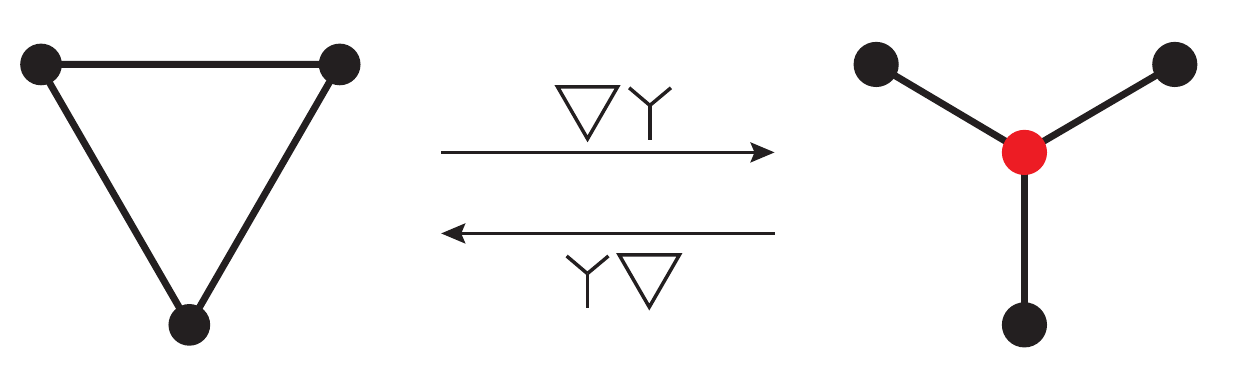}}
\end{picture}
\caption{\small  $\nabla Y-$ and $Y\nabla-$moves} 
\label{fig-ty}
\end{center}
\end{figure}

The  property of being maxnil is, in a way,  analogous to the property of being maximal planar.  
While it is well known that every maximal planar graph with $n$ vertices  has $3n-6$ edges, an analogous statement for maxnil graphs does not exist.
For example, start with a maximal planar graph $G$ and add one vertex $v$ together with all the edges from $v$ to the vertices of $G$. 
Such a graph is maxnil  by \cite{Sa}, and if it has $n$ vertices, then it has $4n-10$ edges. 
In fact, $4n-10$ is an upper bound on the number of edges of a maxnil graph on $n$ vertices.
This follows from work of Mader \cite{Ma} who proved that having more than $4n-10$ edges implies the existence of a $K_6$ minor, which implies the graph is IL.
 
On the other hand,  {J{\o}rgensen \cite{J} and Dehkordi and Farr \cite{DF}} constructed maxnil graphs with $n$ vertices and  $3n-3$ edges.
J{\o}rgensen's maxnil graphs are obtained from the J{\o}rgensen graph in Figure~\ref{FigJ}(a)  by subdividing the highlighted edge incident to the vertex $y$ and then adding edges that connect every new vertex to $u$ and $v$. 
We denote the graph obtained this way through $i$ subdivisions by $J_i$, $i\ge1$.
See Figure~\ref{FigJ}(b).

 
\begin{figure}[htpb!]
\begin{center}
\begin{picture}(400, 150)
\put(0,0){\includegraphics[width=5.4in]{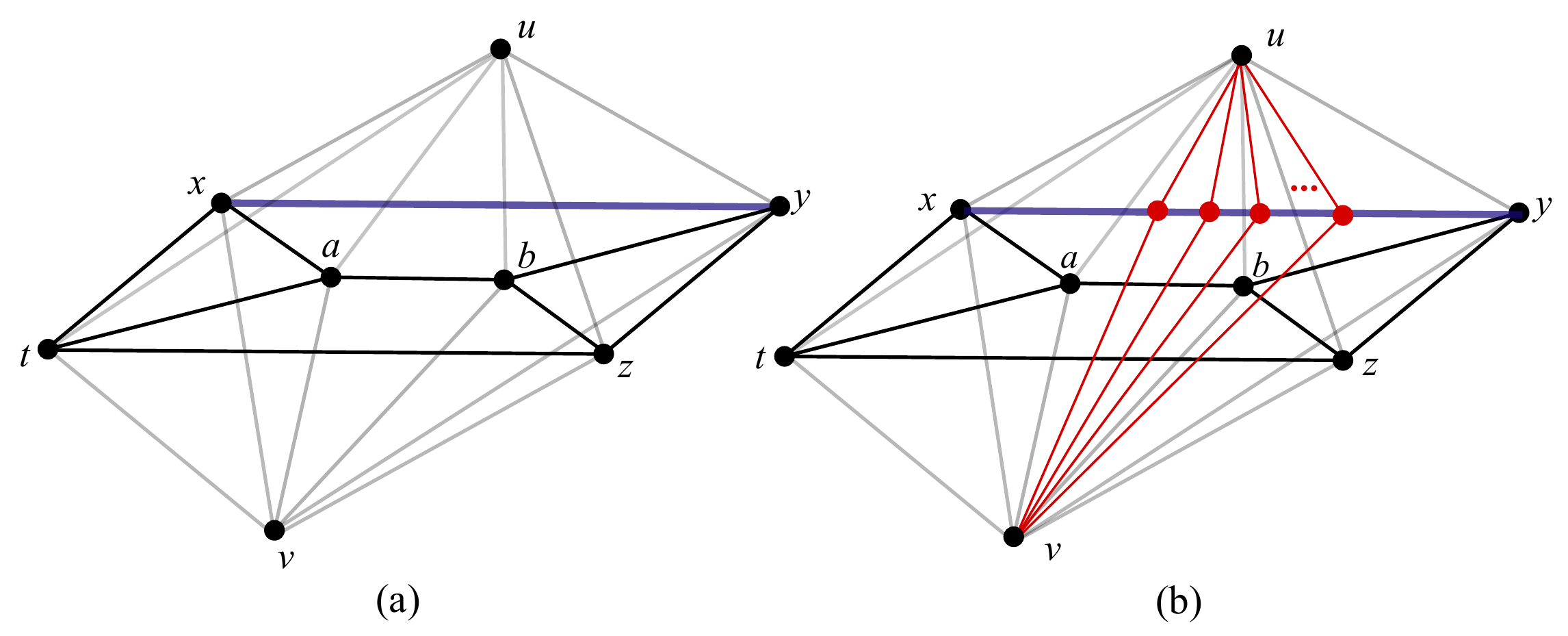}}
\end{picture}
\caption{\small (a) The J{\o}rgensen graph; (b) The graph $J_i$  in J{\o}rgensen's $3n-3$ family.}
\label{FigJ}
\end{center}
\end{figure}

Recently, Aires \cite{A} found a family of graphs with fewer than $3n-3$ edges. 
For each value $n\ge 13$ with $n\equiv3$ (mod 10), he constructed a maxnil graph with $ \frac{14n}{5}-\frac{27}{5}$ edges.
He also proved that  if $G$ is  a maxnil graph with $n \ge 5$ vertices and $m$ edges, then $m\ge 2n$.
This bound is sharp;
the maxnil graph $Q(13,3)$ described by Maharry \cite{M}
has 26 edges and 13 vertices.
 

In Section 2, we present two constructions of maxnil graphs. 
The first one is a family of maxnil graphs with $n\ge 10$  vertices and $3n-5$ edges. 
This construction builds upon a maxnil graph on 10 vertices and 25 edges and uses edge subdivisions. 
The second construction significantly improves on Aires' result on the number of edges.
Using clique sums of copies of $Q(13,3)$, we construct examples with a smaller ``edge-to-vertex ratio,'' as in the following theorem.\\

\textbf{Theorem.}
For each $n\ge 13$, there exists a  maxnil graph $G$ with $n$ vertices and $m <  \frac{25 n}{12} - \frac{1}{4}$ edges.\\

In Section 3, we study the properties of maxnil graphs under  clique sums.  
Some of these results are used in the constructions of Section 2.
We give sufficient and necessary conditions for when the clique sum of two maxnil graphs over $K_2$, $K_3$ or $K_4$ is maxnil.
{ J{\o}rgensen \cite{J} studied clique sums of graphs that are maximal without a $K_6$ minor. We give examples showing that the class of maxnil graphs and the class of graphs that are maximal without a $K_6$ minor are distinct.}

\section{Two families of maxnil graphs}

We note that the J{\o}rgensen graph is 2-apex, i.e., removing the vertices $u$ and $v$ leaves a planar graph $P$. 
Furthermore, the embedding of $P$  in $\mathbb{R}^2$ 
shown in  Figure~\ref{FigJ}(a) has no separating cycles,
i.e., for every cycle $C$ in $P$, one of the components of $\R^2 \setminus C$
contains no vertices of $P$. 
These properties are generalized in the next Lemma, which we use  to prove the graphs in the $3n-5$ family are nIL.

\begin{lemma}
\label{lemma-almost-non-separating}
Let $G$ be a graph with two nonadjacent vertices $u, v$ such that
there exists an embedding $\Sigma$ of $G-\{u,v\}$ in $\R^2$
where for every cycle $C$ in $\Sigma$,
$\R^2 \setminus C$ has a component $X$ such that
$X \cup C$ separates $u$ and $v$ 
(i.e., every path in $G$ from $u$ to $v$ contains a vertex in $X \cup C$).
Then embedding $u$ as $(0,0,1)$ and $v$ as $(0,0,-1)$
and connecting each of them to its neighbors in $\Sigma$ with straight edges
yields a linkless embedding of $G$ in $\R^3$.

\end{lemma}

\begin{proof}
Let $\G$ denote the embedding of $G$ as described in the lemma,
and let $K \cup K'$ be a 2-component link in $\G$.
We consider two cases.

\textit{Case 1.} 
Neither $K$ nor $K'$ contains both $u$ and $v$.
Then we have three subcases:
zero, one, or both of $K$ and $K'$ are in $\Sigma$.
In each of these three subcases it is easy to see that
$K \cup K'$ is a trivial link. 
We prove this for one  of the three subcases here;
the other two are similar and easier.
Suppose $K$ contains $u$ but not $v$,
and $K' \subset \Sigma$.
Then $K$ consists of two edges incident to $u$
and a path $P \subset \Sigma$.
Connecting $u$ with straight line segments to every point in $P$
gives us a $\G$-panel for $K$.
On the other hand, $K'$ bounds a disk $D$ in $\R^2$.
We isotop $D$, while keeping its boundary fixed,
 by pushing its interior slightly below $\R^2$,
to make it disjoint from $K$
(since $K$ contains no points below $\R^2$).
It follows that $K \cup K'$ is a trivial link.

\textit{Case 2.}
One of the link's components, say $K$, contains both $u$ and $v$.
Then $K' \subset \Sigma$.
So $\R^2 \setminus K'$ has two components
such that one of them, $X$, separates $u$ and $v$.
Therefore all vertices of $K$ except $u$ and $v$ lie in $X$.
Now, $K$ has exactly two vertices, call them $a,b$, that are adjacent to $u$,
and two vertices, $c, d$, adjacent to $v$.
Note that $\{a,b\}$ is not necessarily disjoint, or even distinct, from $\{c,d\}$.
Furthermore, $K \cap X$ consists of two components, $P_1$ and $P_2$,
each of which is a path of length zero or greater.
We can assume $a, c \in P_1$ and $b,d \in P_2$.
We consider three subcases.

\begin{figure}[htpb!]
\begin{center}
\begin{picture}(440, 189)
\put(0,0){\includegraphics[width=6in]{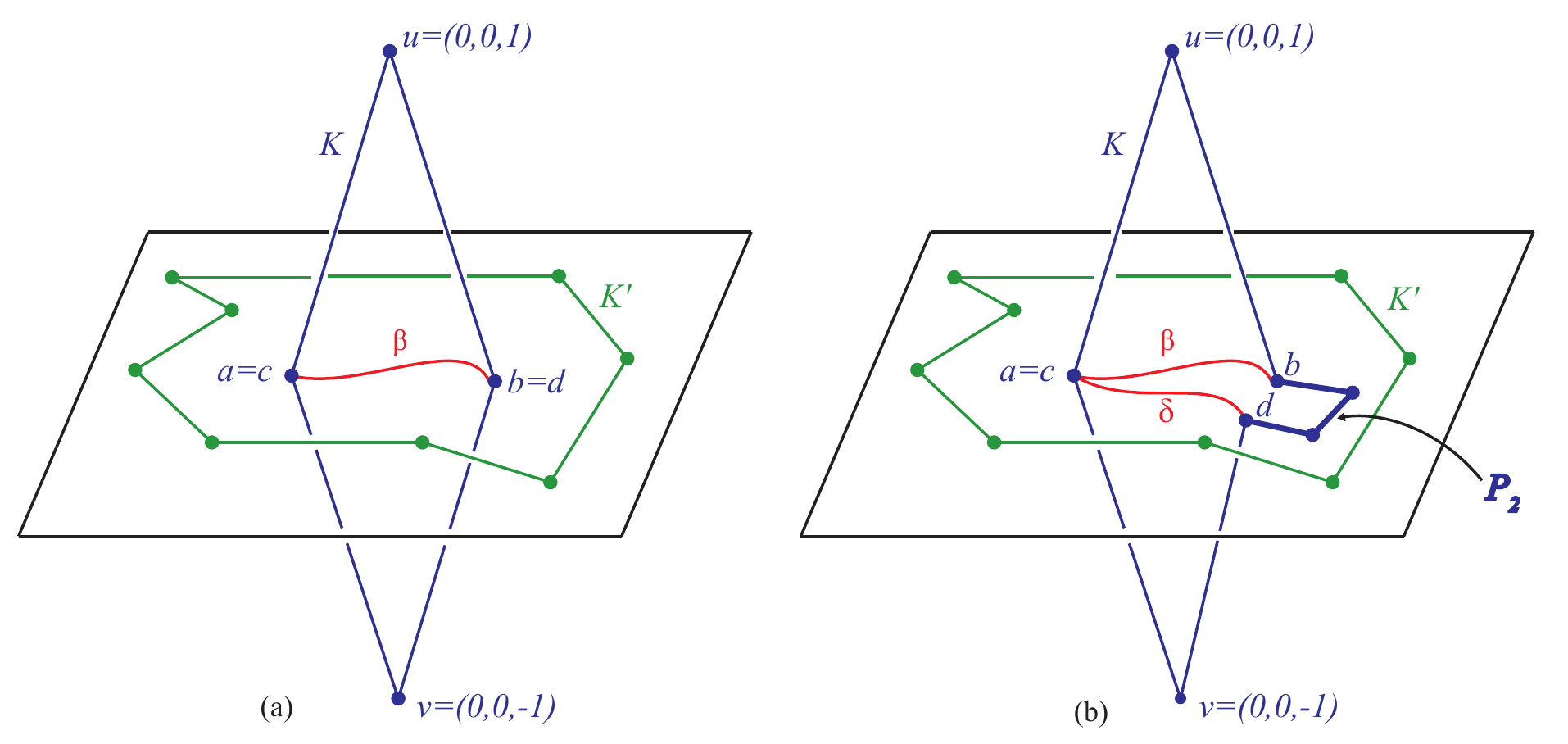}}
\put(423,125){$\mathbb{R}^2$}
\put(206,125){$\mathbb{R}^2$}
\end{picture}
\caption{\small (a) Configuration for Case 2.1 ; (b) Configuration for Case 2.2} 
\label{fig-almost-nonsepai}
\end{center}
\end{figure}

\textit{Case 2.1.}
$a=c$, $b=d$.
Join $a$ to $b$
by an arc $\beta \subset X$ (not necessarily in $\Sigma$),
and then connect each of $u$ and $v$
by straight line segments to every point in $\beta$.
See Figure~\ref{fig-almost-nonsepai}(a).
This gives us a disk bounded by $K$ and disjoint from $K'$.
Similarly to Case~1 above, $K'$ also bounds a disk disjoint from $K$.
Hence $K \cup K'$ is a trivial link.

\textit{Case 2.2.}
$a=c$, $b \ne d$.
Join $a$ to each of $b$ and $d$
by disjoint arcs $\beta$ and $\delta$ respectively, both in $X$,
such that $\beta \cup \delta \cup P_2$ is a simple closed curve.
See Figure~\ref{Fig1025}(b).
 Connect each of $u$ and $v$ by straight line segments 
to every point in $\beta$ and $\delta$ respectively.
This gives us two disks whose union with 
the disk bounded by $\beta \cup \delta \cup P_2$ in $X$
is a disk bounded by $K$ and disjoint from $K'$.
And, as before, $K'$ bounds a disk disjoint from $K$.
Hence $K \cup K'$ is a trivial link.

\textit{Case 2.3.}
$a \ne c$, $b \ne d$.
This case is similar to Case~2.2, except that
we join $a$ to $b$ and $c$ to $d$
by disjoint arcs $\beta$ and $\delta$ in $X$
such that $\beta \cup \delta \cup P_1  \cup P_2$ 
is a simple closed curve.

\end{proof}

\subsection{The $3n-5$ family} We construct a family of graphs with $n$ vertices and $3n-5$ edges, for $n\ge 10$.
This family is obtained from the graph $G$ pictured in Figure~\ref{Fig1025}(a) through a sequence of subdivisions and edge additions.
The graph $G$  is obtained from the  J{\o}rgensen graph by splitting (the opposite of contracting edges) the vertices  $a$ and $b$  into the edges $ad$ and $bc$. 
See Figures~\ref{FigJ}(a) and~\ref{Fig1025}(a).
With the notation in Figure~\ref{Fig1025}(a), construct the graph $G_1$ by subdividing the edge $xy$ with a new vertex $z_1$, then adding edges $z_1u$ and $z_1v$.
Construct  graphs $G_i$, for $i\ge 2$, as follows:  subdivide the edge $z_{i-1}y$ of $G_{i-1}$ with a new vertex $z_i$, then add edges $z_iu$ and $z_iv$ to $G_{i-1}$.  
Notice that $G_i$ has one more vertex and three more edges than $G_{i-1}$.
The graph $G_i$ has $10+i$ vertices and $25+3i= 3(10+i) -5$ edges.
We note that the graphs $G_i$  can also be
obtained by successive splittings of the vertex $y$ 
 into the edge $y z_i$.

 
\begin{figure}[htpb!]
\begin{center}
\begin{picture}(440, 170)
\put(0,0){\includegraphics[width=6in]{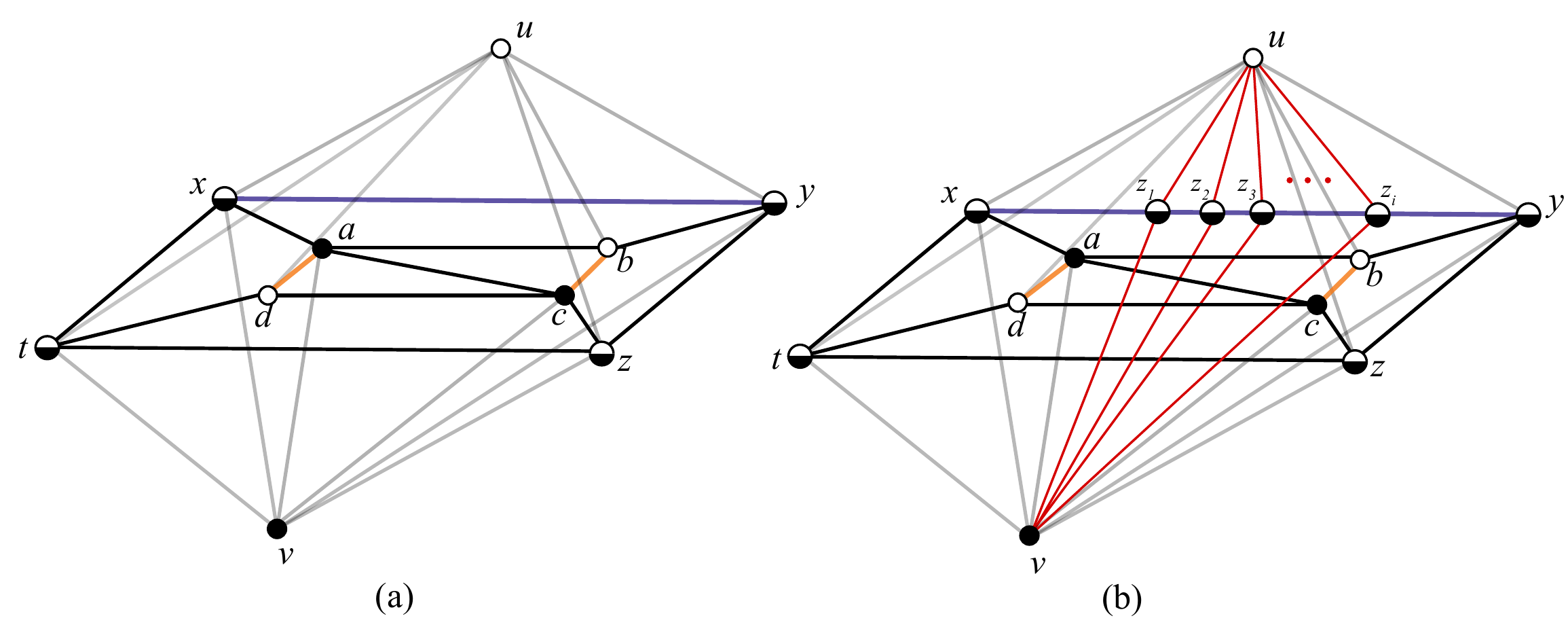}}
\end{picture}
\caption{\small (a) The graph $G$ is maxnil with 10 vertices and 25 edges; (b) The graph $G_i$ is obtained through $i$ edge subdivisions and edge additions.}
\label{Fig1025}
\end{center}
\end{figure}

\begin{proposition} The graphs $G$ and $G_i$ in Figures~\ref{Fig1025}(a)  and~\ref{Fig1025}(b) are linklessly embeddable.
\label{Gnil}
\end{proposition}

\begin{proof}
It is straightforward to check that 
these graphs satisfy the hypotheses of Lemma~\ref{lemma-almost-non-separating}
and hence are nIL.
\end{proof}\

\begin{proposition} The graph $G$ in Figure~\ref{Fig1025}(a) is maxnil.
\label{Gmaxnil}
\end{proposition}
\begin{proof} Since $G$ is linklessly embeddable, it remains to show that adding any edge to $G$ gives an IL graph.
We note that by contracting the edges $ad$ and $bc$, we obtain the  J{\o}rgensen graph as a minor of $G$.
If an edge $e$ other than $bd$ is added to $G-\{u,v\}$, such edge  is preserved by these two  edge contractions. 
Thus $G+e$ contains a minor that itself contains the J{\o}rgensen graph added an edge. 
Since the J{\o}rgensen graph is maxnil, $G+e$ is IL.
Same holds is $e=uv$ is added to $G$.
If the edge $bd$ is added, then contracting the edges $dt$, $cz$, $ux$ and $vy$  creates a $K_6$ minor of $G+bd$.
If an edge from $u$ to $G-\{u,v\}$ is added, say $ua$, then contracting the edges $cd$, $dt$, $by$ and $uz$ creates a $K_6$ minor of $G+ua$.
If an edge from $v$ to $G-\{u,v\}$ is added, say $vb$, then contracting the edges $ax$, $cz$, $du$ and $dt$ creates a $K_6$ minor of $G+vb$.
\end{proof}

\begin{proposition} All graphs $G_i$, $i\ge 1$ are maxnil.
\end{proposition}
\begin{proof} Since $G_i$ is linklessly embeddable, it remains to show that adding any edge to $G_i$ gives an IL graph.
Adding any edge $e$ different from $xy$ and disjoint from $\{z_1, z_2, \ldots, z_i\}$ to $G_i$ gives a graph $G_i+e$ that contains $G+e$ as a minor (obtained by contracting the path $xz_1z_2...z_i$). 
Since $G$ is maxnil, $G+e$ is IL and so is $G_i+e$.
Adding an edge $e$ that is either $xy$ or has at least one endpoint in $\{z_1, z_2, \ldots, z_i\}$ to $G_i $, gives a graph $G_i+e$ that contains $J_i+e$ as a minor (obtained by contracting the edges $ad$ and $bc$). 
Since $J_i$ is maxnil, $J_i+e$ is IL and so is $G_i+e$.

\end{proof}

\subsection{The $Q(13,3)$ family} A graph $G$ is called \textit{triangular} if each edge of $G$ belongs to at least a triangle. 
In a non-triangular graph, an edge that is not part of a triangle is a \textit{non-triangular edge.}
In Section 3, we study the properties of maxnil graphs under the operation of clique sum (defined in Section 3).
For the  construction presented in the next theorem we use the result of Lemma~\ref{lemmajoin2} about clique sums of maxnil graphs over $K_2$.

\begin{theorem}
For each $n\ge 13$, there exists a  maxnil graph $G$ with $n$ vertices and $m< \frac{25 n}{12} - \frac{1}{4}$ edges.
\label{main}
\end{theorem}

\begin{proof}
The construction is based on the maxnil graph $Q_{13,3}$ described by Maharry \cite{M}. See Figure~\ref{Q133}(a).
This graph has 13 vertices and 26 edges, and it is triangle free.

For each $n$ with $13\le n\le 39$,  we construct a set of maxnil graphs with $n$ vertices and $2n$ edges by
adding $n-13$ new vertices, and then
choosing $n-13$ edges in $Q_{13,3}$
and connecting the two endpoints of each of them
to one of the new vertices.
Equivalently, we are taking the clique sum
of $Q_{13,3}$ with $n-13$ disjoint triangles,
over $n-13$ copies of $K_2$.
See Figure~\ref{Q133}(b).
By Lemma \ref{lemmajoin2}, the resulting graph is maxnil.

\begin{figure}[htpb!]
\begin{center}
\begin{picture}(330, 155)
\put(0,0){\includegraphics[width=4.8in]{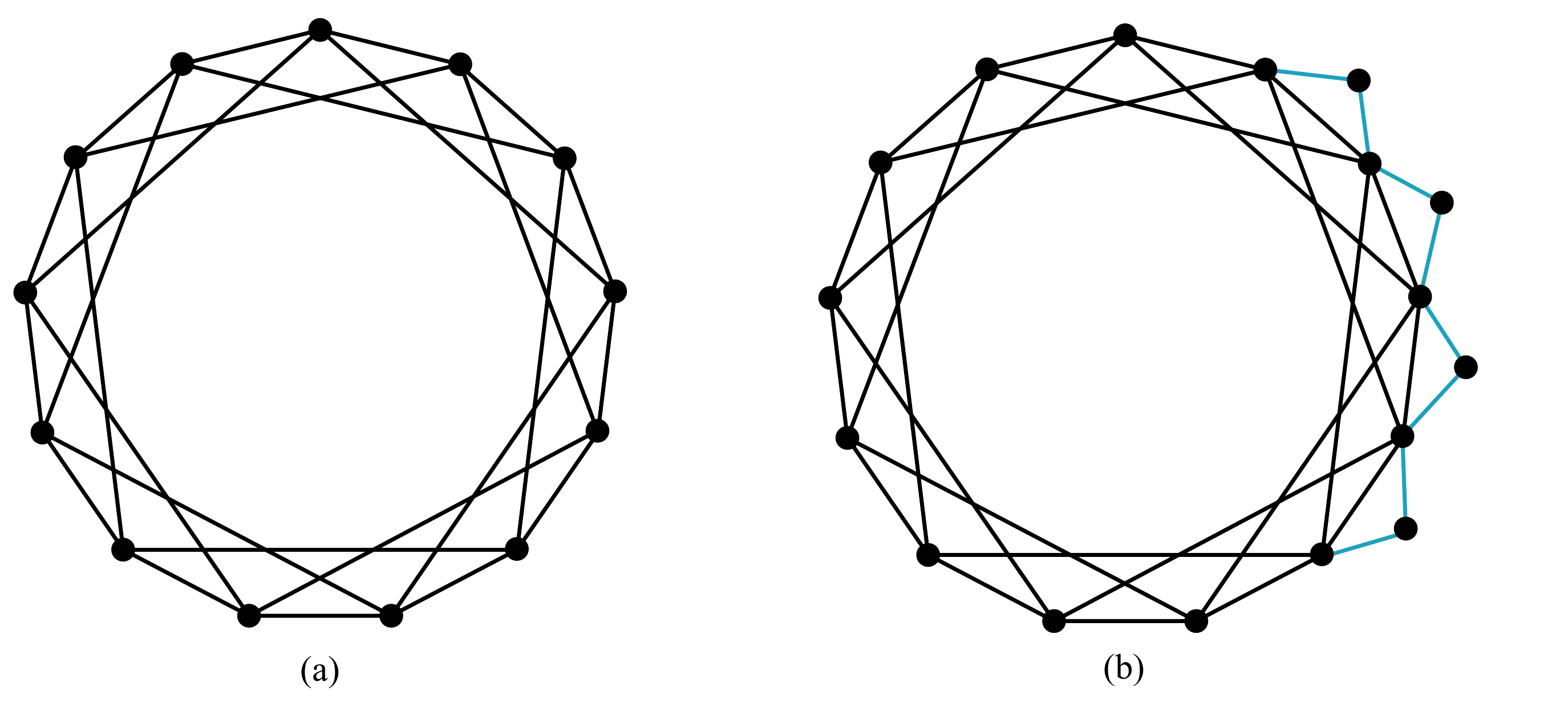}}
\end{picture}
\caption{\small (a) $Q_{13,3}$ is a maxnil graph with 13 vertices and 26 edges; (b) A maxnil graph with 17 vertices and 34 edges obtained from  $Q_{13,3}$ by adding four vertices of degree 2 and eight edges.}
\label{Q133}
\end{center}
\end{figure}

The graph on 39 vertices obtained this way is triangular, so the construction cannot proceed further. 
To build graphs with a larger number of vertices we use multiple copies of $Q_{13,3}$ joined along an edge (clique sum over $K_2$). 
Consider  $k\ge 1$ copies of $Q_{13,3}$ and choose one edge in each copy. 
Then join the $k$ graphs  together by identifying the
$k$ chosen edges into one edge.
This graph, which we denote by $H_k$, is maxnil  (by repeated application of Lemma~\ref{lemmajoin2}) and has $11k+2$ vertices and $25k+1$ edges. 
All edges of $H_k$ are non-triangular and adding vertices of degree 2 (as above) along any subset of the edges of $H_k$ gives a maxnil graph.

For $ n \ge 13$, let $k=\lceil \frac{n-3}{36}\rceil$ and 
add $n-(11k+2)$ vertices of degree 2 along any $n-(11k+2)$ edges of $H_k$.
With  every added vertex of degree 2, the number of edges is increased by 2.
This gives a maxnil graph with $n$ vertices and  $m=(25k+1)+2[n-(11k+2)] = 2n+3k-3$ edges.
Moreover,
$$m=2n+ 3 \lceil \frac{n-3}{36}\rceil-3 <2n +3 ( \frac{n-3}{36}+1)-3 =\frac{25n}{12}-\frac{1}{4}.$$
\end{proof}
\begin{remark}

The above shows there exist maxnil graphs of arbitrarily large order $n$
with an edge-to-vertex ratio of less than $\frac{25}{12} - \frac{1}{4n}$.
Whether this  edge-to-vertex ratio can be lowered further is an open question.  

\end{remark}

\section{Clique Sums of  Maxnil Graphs}

In this section we study the properties of maxnil graphs under taking clique sums. 
A set  $S \subset V(G)$ is a \dfn{vertex cut set} of a connected graph $G$ if  $G-S$ is disconnected.
 We say a vertex cut set $S \subset V(G)$ is \dfn{minimal} if no proper subset of $S$ is a vertex cut set of $G$.
A graph $G$ is the \textit{clique sum} of $G_1$ and $G_2$ over $K_t$ if $V(G)=V(G_1)\cup V(G_2)$, $E(G)=E(G_1)\cup E(G_2)$, and the subgraphs induced by $V(G_1)\cap V(G_2)$ in both $G_1$ and $G_2$ are complete of order $t$. 
Since the vertices of the clique over which a clique sum is taken is a vertex cut set in the resulting graph, the vertex connectivity of a clique sum over $K_t$ is at most $t$.
 For a set of vertices $\{v_1, v_2, \ldots, v_k\} \subset V(G)$, $\big< v_1, v_2, \ldots, v_k\big>_G$ denotes the subgraph of $G$  induced by this set of vertices.
 By abuse of notation, the subgraph induced in $G$ by the union of the vertices of subgraphs $H_1, H_2, \ldots, H_k$ is denoted by $\big< H_1,H_2, \ldots, H_k\big>_G$.
 
Holst, Lov\'asz, and Schrijver \cite{HLS} studied the behavior of the Colin de Verdi\'ere $\mu-$invariant for graphs under clique sums (Theorem 2.10).
Since a graph $G$ is nIL if and only if $\mu(G)\le 4$ (\cite{LS}, \cite{RST2}), their theorem implies the following.
\begin{theorem}[Holst, Lov\'asz, and Schrijver \cite{HLS}]If $G$ is the clique sum over $S$ of two nIL graphs, then $G$ is IL if and only if one can contract two or three components of $G-S$ so that the contracted nodes together with $S$ form a $K_7$ minus a triangle.
\label{ThmHLS}
\end{theorem}
Theorem~\ref{ThmHLS} implies that for $t\le 3$, the clique sum over $K_t$ of nIL graphs is nIL.
While Theorem~\ref{ThmHLS} shows when a clique sum is nIL,  it does not establish when a clique sum of maxnil graphs is maxnil. 

 \begin{lemma} Any maxnil graph is 2-connected.
\label{lemma-2-connected}
\end{lemma}

\begin{proof} Let $G$ be a maxnil graph. If $G$ is disconnected, let $A$ and $B$ denote two of its connected components. Let $a\in V(A)$ and $b\in V(B)$. Then $G+ab$ is a nIL graph, as it can be obtained by performing two consecutive clique sums over $K_1$ of nIL summands, namely 
\[G+ab=A\cup_{\{a\}}ab\cup_{\{b\}}(G-A).\] But this contradicts the maximality of $G$.

If the vertex connectivity of $G$ is one, assume $x\in V(G)$ is a cut vertex, that is $G-x=A\sqcup B$, with $A$ and $B$ nonempty, and no edges between vertices of $A$
and vertices of $B$.  Let $a\in V(A)$ and $b\in V(B)$ be neighbors of $x$ in $G$.
Then $G+ab$ is nIL, as it can be obtained by performing two consecutive clique sums over $K_2$ of nIL summands.
 If $\Delta$ denotes the triangle $axb$, 
\[G+ab=\big<A, x\big>_G\cup_{ax}\Delta\cup_{xb}\big<B, x\big>_G.\] But this contradicts the maximality of $G$.
\end{proof}

\begin{lemma} Let $G$ be a maxnil graph with a vertex cut set $S=\{x,y\}$, and let $G_1,G_2,...,G_r$ denote the connected components of $G-S$. 
Then
$xy\in E(G)$ and
$\big<G_i, S \big>_G$ is maxnil for all $1\le i \le r$.
\label{cliquesum2}
\end{lemma}
\begin{proof} 
By Lemma~\ref{lemma-2-connected}, $x$ and $y$ are distinct
and each of them has at least one neighbor in each $G_i$.
Suppose $xy \not \in E(G)$.
Let $G' = G + xy$ and $G'_i = \big<G_i, S\big>_{G'}$.
Then, for every $i$, $G'_i$ is a minor of $G$
since if we pick a $j \ne i$ and
in $\big<G_i, G_j, S \big>_G$ contract $G_j$ to $x$,
we get a graph isomorphic to $G'_i$.
So $G'_i$ is nIL.
Then, by  Theorem~\ref{ThmHLS},
$G' = G'_1 \cup_{xy} \cdots \cup_{xy} G'_r$ is nIL,
contradicting the assumption that $G$ is maxnil.
So $xy \in E(G)$.

For each $i$,
we repeatedly add new edges to $\big<G_i, S \big>_G$, if necessary,
to get a maxnil graph $H_i$.
Then $H := H_1 \cup_{xy} \cdots \cup_{xy} H_r$ is nIL
and contains $G$ as a subgraph,
so $H = G$ and every $\big<G_i, S \big>_G$
is maxnil.
\end{proof}

\begin{lemma} 
Let $G_1$ and $G_2$ be maxnil graphs.
Pick an edge in each $G_i$ and label it $e$.
Then $G=G_1\cup_e G_2$ is maxnil
if and only if
$e$ is non-triangular in at least one $G_i$.
\label{lemmajoin2}
\end{lemma}
\begin{proof} 

The graph $G$ is nIL by Theorem~\ref{ThmHLS}.
Suppose $e$ is non-triangular in at least one $G_i$, say $G_2$.
To prove $G$ is maxnil, it is enough to show that
for all $b_i \in V(G_i)$,
$G + b_1 b_2$ is IL.
Denote the endpoints of $e$ in $G$ by $x, y$.
By Lemma~\ref{lemma-2-connected}, $G_1$ is 2-connected,
so each of $x, y$ has at least one neighbor in $G_1$.
So if we contract $G_1$ to $b_1$ and then contract $b_1 b_2$ to $b_2$,
we obtain a graph $G'_2$ that 
contains $G_2$ as a proper subgraph since
$b_2 x \in E(G'_2)$.
So $G'_2$ is IL since $G_2$ is maxnil.
But $G'_2$ is a minor of $G$, which is nIL,
so we have a contradiction.

To prove the converse, suppose $e $ is triangular in  $G_1$ and $G_2$.
Let $t_i \in V(G_i)$ be adjacent to both endpoints of $e$.
Let $K$ be a complete graph on four vertices,
with vertices labeled $x, y, t_1, t_2$.
Denote the triangles induced by $x, y, t_i$
in $K$ and in $G_i$ by $\Delta_i$.
Then by Theorem~\ref{ThmHLS},
$G':= G_1 \cup_{\Delta_1} K  \cup_{\Delta_2} G_2$ is nIL.
But $G'$ is isomorphic to $G + t_1 t_2$,
so $G$ is not maxnil.

\end{proof}

\begin{lemma} 
Let $G$ be a maxnil graph with vertex connectivity 3 and a vertex cut set $S=\{x,y,z\}$. 
Let $G_1,G_2,...,G_r$ denote the connected components of $G-S$.
Then  $\big<S\big>_G\simeq K_3$ and
$\big<G_i, S \big>_G$ is maxnil for all $1\le i \le r$.
\label{cliquesum3}
\end{lemma}
\begin{proof}
Suppose $\big<S\big>_G \not \simeq K_3$.
Let $G'$ be the graph obtained from $G$ 
by adding one or more edges to $\big<S\big>_G$ so that 
$S$ induces a triangle $T$ in $G'$.
For $1 \le i \le r$, let $G'_i = \big< G_i, T \big>_{G'}$.
We see as follows that $G'_i$ is nIL.
Pick any $j \ne i$, and in the graph $\big< G_i , G_j , S \big>_G$,
contract $G_j$ to an arbitrary vertex $v$ in $G_j$.
Then $v$ is connected to each of $x, y, z$
since $G$ is 3-connected
and hence  each of $x, y, z$ has at least one neighbor in $G_j$.
The graph $M_i$ obtained this way is minor of $G$, and hence is nIL.
Performing a $\nabla Y$-move on $T\subset G_i'$ we obtain a subgraph of $M_i$. 
Since $M_i$ is nIL, so is $G_i'$.
By Theorem~\ref{ThmHLS},
$G' = G'_1 \cup_{T} \cdots \cup_{T} G'_r$ is nIL,
which contradicts the maximality of $G$.
So  $T = \big<S\big>_G\simeq K_3$.

To show $\big<G_i, S \big>_G$ is maxnil, 
repeatedly add new edges to it, if necessary,
to get a maxnil graph $H_i$.
Then $H := H_1 \cup_{T} \cdots \cup_{T} H_r$ is nIL by Theorem~\ref{ThmHLS} and contains $G$ as a subgraph,
so $H = G$ and every $\big<G_i, S \big>_G$
is maxnil.
\end{proof}

Let $G$ be a  graph and let $T=\big<x,y,z,t\big>_G$ be an induced $K_4$ subgraph (\dfn{tetrahedral graph}).
 We say $T$ is \textit{strongly separating} 
if $G-T$ has at least  two connected components $C_1, C_2$
such that every vertex of $T$ has a neighbor in each $C_i$.

\begin{lemma}Let $G_1$, $G_2$ be maxnil graphs and let $G=G_1\cup_{\triangle} G_2$ be the clique sum of $G_1$ and $G_2$ over a  $K_3$ subgraph $\Delta=\big<x,y,z\big>_G$.
Assume  $\Delta$ is a minimal vertex cut  set in $G$.
Then 
$G$ is maxnil if and only if  for some $i \in \{1,2\}$, every induced  $K_4$ subgraph of the form $\big<x,y,z,t \big>_{G_i}$ is strongly separating.

\label{lemmajoin3}
\end{lemma}

\begin{proof}
By Theorem~\ref{ThmHLS}, $G := G_1 \cup_{\Delta} G_2$ is nIL. 
Then $G$ is maxnil if and only if  for every  $t_1\in V(G_1)-V(\Delta)$ and $t_2\in V(G_2)-V(\Delta)$, $G' :=G+t_1t_2$ is IL.

First, suppose for some $i$ at least one of $x,y,z$ is not connected to $t_i$,
say $x t_2 \notin E(G_2)$. 
Contracting $G_1 - \{y,z\}$ to $x$ produces $G_2+t_2x$ as a minor of $G'$.
Since $G_2$ is maxnil, this minor is IL, and hence $G'$ is IL, as desired. 
So we can assume $\big<x,y,z,t_i \big>_{G_i}$ is a  tetrahedral graph for both $i = 1,2$.

Assume every tetrahedral graph in $G_2$ that contains $\Delta$ is strongly separating.
So $G_2-\big<x,y,z,t_2\big>_{G_2}$ has at least two connected components each of which, when contracted to a single vertex, is adjacent to all four vertices $x, y, z, t_2$. 
In Figure~\ref{K3clique} these vertices are denoted by $c_1$ and $c_2$.
Now, if  the component of $G_1-\Delta$ that contains $t_1$ 
is contracted to $t_1$,
this vertex too will be adjacent to $x, y, z, t_2$.
So we get a minor of $G$ isomorphic to $K_7$ minus a triangle, 
which is IL since it contains a Petersen family graph
(the one obtained by one $\nabla Y$-move on $K_6$) as a minor. 
It follows that $G'$ is IL, and therefore $G$ is maxnil.

\begin{figure}[htpb!]
\begin{center}
\begin{picture}(200, 120)
\put(0,0){\includegraphics[width=3in]{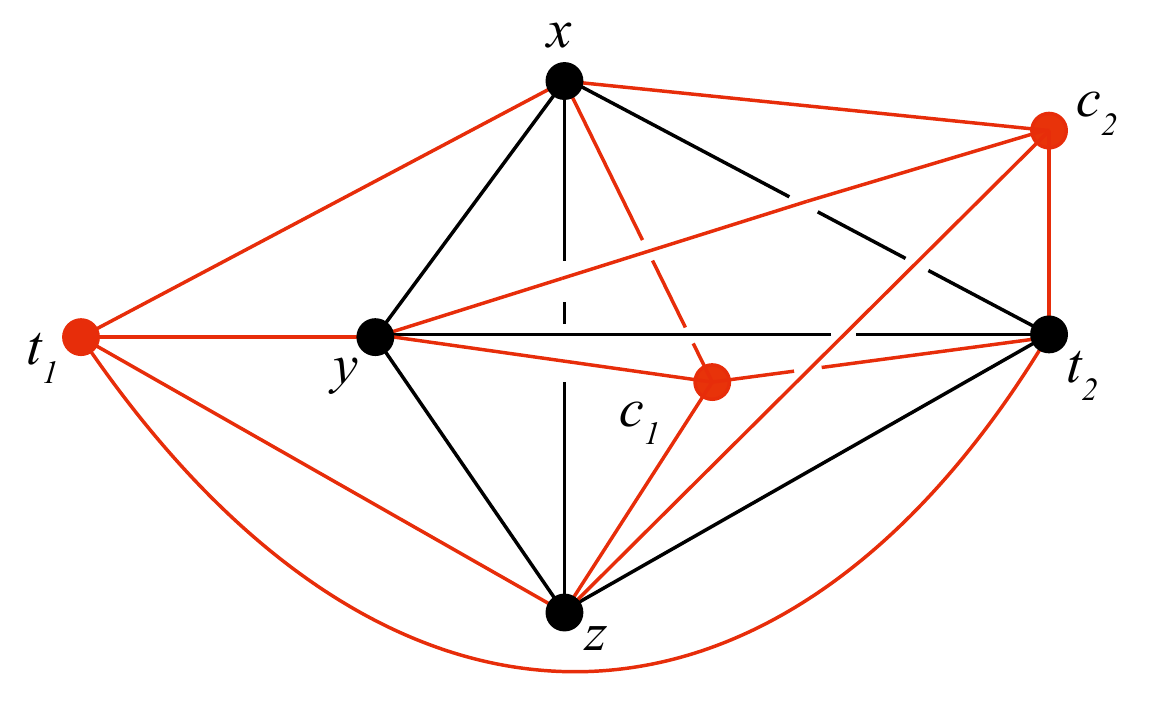}}
\end{picture}
\caption{\small A $K_7$ minus a triangle minor of the graph $G$.}
\label{K3clique}
\end{center}
\end{figure}

To prove the converse, 
for $i=1,2$ let $t_i$ be a vertex in $G_i$ such that 
$T_i := \big<x,y,z,t_i\big>_{G_i}$ is a tetrahedral graph
that is not strongly separating.  
Let $G'=G+t_1t_2$.
Then $G' =  
G_1\cup_{T_1} \big<x,y,z,t_1,t_2\big>_{G'}\cup_{T_2} G_2$.
Each of these click sums is over a $K_4$, 
each summand is nIL, and each of $T_1, T_2$ is  non-strongly separating;
so, by Theorem~\ref{ThmHLS}, $G'$ is nIL, and hence $G$ is not maxnil.

\end{proof}

 Unlike the vertex connectivity 2 and 3 cases, it is not true that a minimal vertex cut set in a 4-connected maxnil graph must be a clique. 
 The four neighbors of $b$ in the graph depicted in Figure~\ref{Fig1025}(a) form a vertex cut set, but the graph induced by its vertices has exactly 2 edges. 
 The four neighbors of any vertex in the graph $Q_{13,3}$ in Figure~\ref{Q133}(a) form a discrete vertex cut set. 
 However, if a maxnil graph $G$ has vertex connectivity 4, the following lemma provides some restrictions on the shape of the subgraph induced by the vertices of any minimal vertex cut set.

 \begin{lemma} Let $G$ be a maxnil graph and assume $\{x,y,z,t\}$ is a minimal vertex cut. 
Let $S=  \big<x,y,z,t\big>_G$.
Then $S$ is either a clique or a subgraph of $C_4$ (a 4-cycle).
\label{vertexcut4}
\end{lemma}

\begin{proof}
Assume that $S$ is neither a clique nor a subgraph of a 4-cycle. 
This implies that
if $S$ has a vertex of degree at least 3, then it contains $K_{1,3}$ as a subgraph;
and if every vertex of $S$ has degree less than 3, then $S$ contains $K_3$ as a subgraph.
Below, we consider  these two cases separately.

\textit{Case 1.}
$S$ has a $K_3$ subgraph.
We can assume that $x, y, z$ induce a triangle in $G$. If $G-S$ has at least three connected components, contracting each of them to a single node would produce a minor of $G$ isomorphic to $K_7$ minus a triangle, contradicting that $G$ is nIL.
So $G-S=G_1\sqcup G_2$, with $G_1$ and $G_2$ each connected.

Since $\{x,y,z,t\}$ is a minimal vertex cut set in $G$,
each of $x, y, z, t$ has at least one neighbor in each $G_i$, $i=1,2$.
Contracting $\big<G_i,t\big>_G$ to $t$ produces a minor of $G$, denoted by $G'_i$, which must be nIL.
Since $xyz$ is a triangle and each of $x, y, z$ has at least one neighbor in each $G_i$,
$\{x,y,z,t\}$ induces a clique $T$  in both $G'_1$ and $G'_2$.
By \cite{HLS}, the clique sum $G'=G'_1\cup _{T} G'_2$  is nIL since $G'-T=G_1\sqcup G_2$; but $G'$ strictly contains $G$ as a subgraph, a contradiction.

\textit{Case 2.}
$S$ has a $K_{1,3}$ subgraph.
We can assume that $t$ is adjacent to $x$, $y$ and $z$ in $G$. If $G-S$ has at least three connected components, contracting each of them to a single node would produce a minor of $G$ containing a subgraph isomorphic to $K_{3,3,1}$, thus $G$ is IL. 
So $G-S=G_1\sqcup G_2$, with $G_1$ and $G_2$ connected.
For $i=1,2$, contracting each of $G_i$ to a single node $t_i$, deleting the edge $t_it$, deleting any existing edges of $\big<x,y,z \big>_G$, and then performing a $Y\nabla$-move at $t_i$, produces a nIL graph, denoted by $G'_i$. 
Let $G'=G'_1\cup _{K_4}G'_2$ be the clique sum 
over the complete graph with vertices $x,y,z,t$. 
By Theorem~\ref{ThmHLS}, $G'$ is nIL since $G'-S=G_1\sqcup G_2$; 
but $G'$ strictly contains $G$ as a subgraph, a contradiction.

\end{proof}
\begin{lemma} Let $G=G_1\cup_SG_2$ be the clique sum of maxnil graphs $G_1$ and $G_2$ over $S=\big<x,y,z,t\big>_G\simeq K_4$. 
Assume $S$ is a minimal vertex cut set in $G$.
Then $G$ is maxnil if and only if,  in both $G_1$ and $G_2$, $S$ is not strongly separating.
\label{lemmajoin4}
\end{lemma}

\begin{proof}
If $S$ is strongly separating in $G_1$ or $G_2$, then $G-S$ has at least three connected components and contracting each of them to a single node produces a minor isomorphic to $K_7$ minus a triangle.

If, in both $G_1$ and $G_2$, $S$ is not strongly separating, 
then $\G - S$ has only two connected components.
Contracting each of the two components to a single node produces $K_6$  minus an edge as a minor (not $K_7$ minus a triangle); hence $G$ is nIL by Theorem~\ref{ThmHLS}. 
Adding an edge between a vertex in $G_1-S$ and a vertex $G_2-S$ and contracting $G_1-S$ and $G_2-S$ to single nodes produces a $K_6$ minor. 
It follows that $G$ is maxnil in this case.
\end{proof}

\begin{figure}[htpb!]
\begin{center}
\begin{picture}(200, 120)
\put(0,0){\includegraphics[width=3.3in]{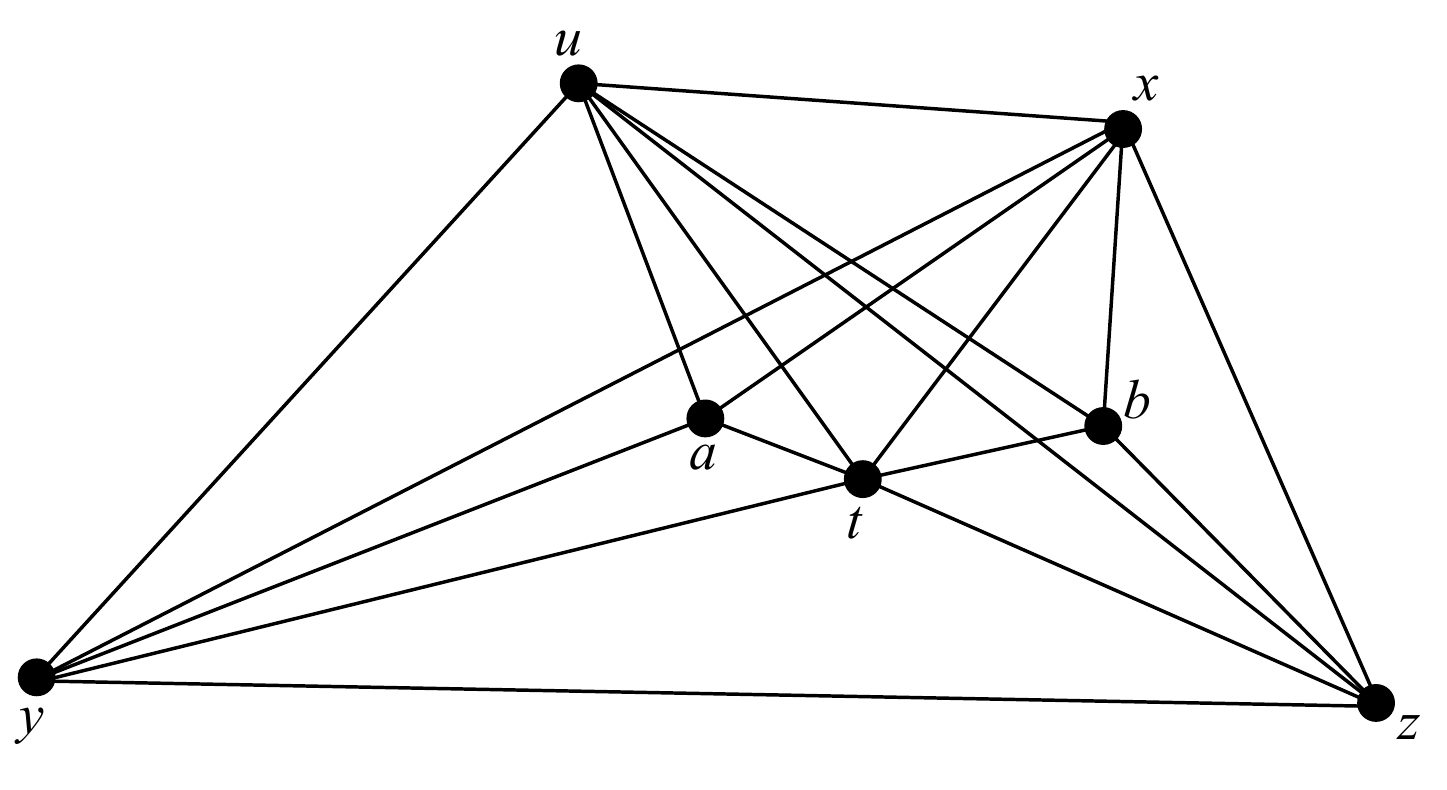}}
\end{picture}
\caption{\small A maxnil graph that is a clique sum over $K_5$ .}
\label{K5clique}
\end{center}
\end{figure}

The graph $G$ of Figure~\ref{K5clique} is maxnil since $G - u$ is a maximal planar graph. If S=$\big<x, y, z, t, u\big>$, $G_1=\big<a,x, y, z, t, u\big>$, and $G_2=\big<b,x, y, z, t, u\big>$, then $S \simeq K_5$, $G_1\simeq G_2 \simeq K_6^-$ ($K_6$ minus one edge), and $G=G_1\cup_S G_2$. 
This shows it is possible for the clique sum of two maxnil graphs over $S\simeq K_5$ to be nIL (and maxnil). 
However,  no clique $S$ of order 5 can be a minimal vertex cut set in a nIL graph $G$,  since then any connected component of $G-S$ would form a $K_6$-minor together with $S$, which would imply $G$ is IL.
For $t \ge 6$, any clique sum over $K_t$ is IL since $K_6$ is IL.

\vspace*{.1in}

J{\o}rgensen  studied clique sums of graphs that are maximal without a  $K_6$  minor \cite{J}.
These are graphs that do not contain a $K_6$ minor and a $K_6$ minor is created by the addition of any edge.
The class of maxnil graphs and the class of graphs that are maximal without a  $K_6$  minor are not the same, as shown in the following proposition.


\begin{figure}[htpb!]
\begin{center}
\begin{picture}(400, 200)
\put(0,0){\includegraphics[width=5.8in]{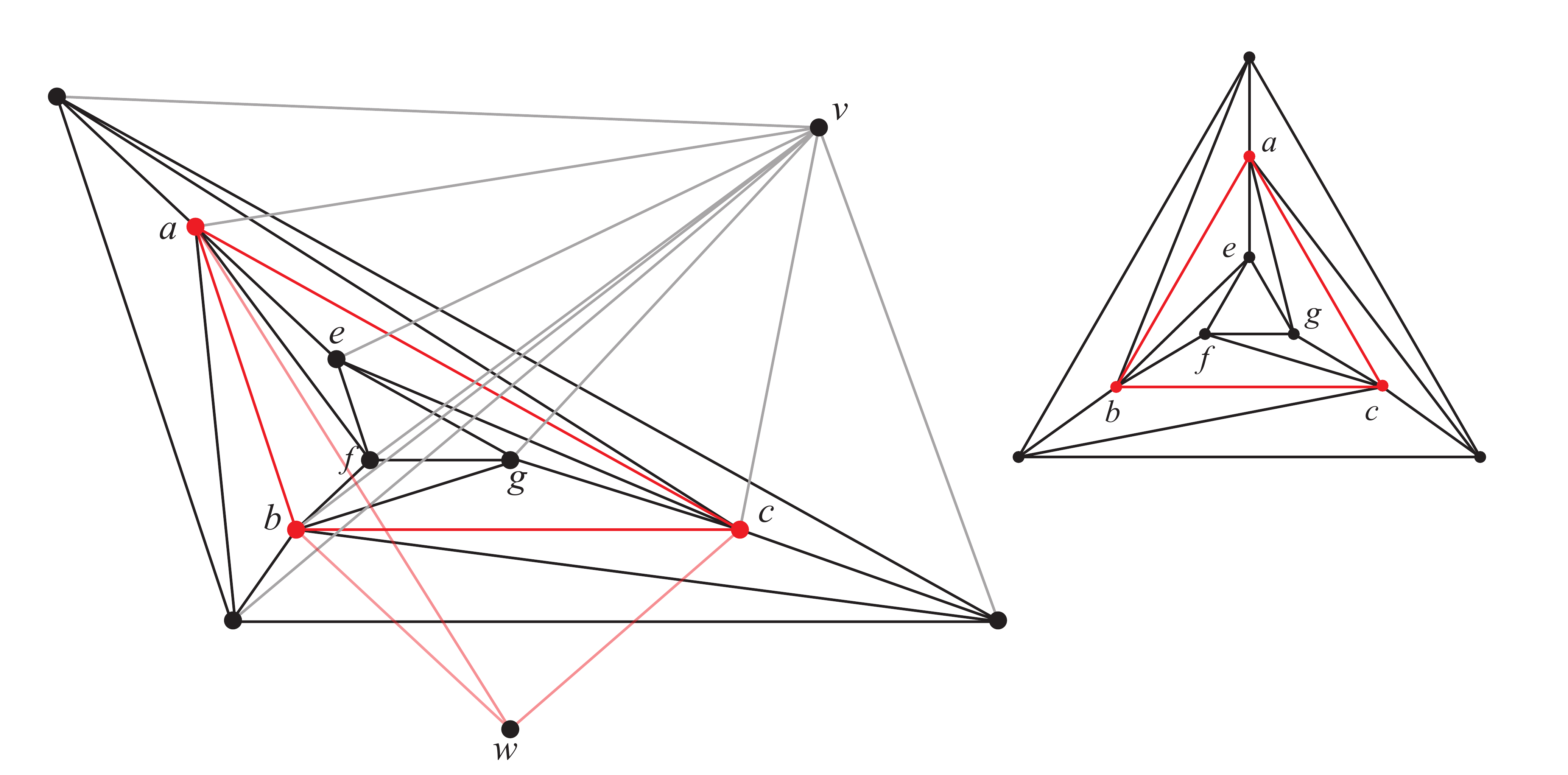}}
\put(300,70){\small $H=G\setminus\{v,w\}$}
\end{picture}
\caption{\small  {A maxnil graph $G$ (left) that is not maximal without a  $K_6$ minor is obtained by adding two  vertices to a plane triangulation with nine vertices (right)}}
\label{fig-telescope}
\end{center}
\end{figure}

{ \begin{proposition}
 The graph in Figure \ref{fig-telescope} is maxnil, and it is not maximal without a $K_6$ minor.
 \label{noK6}
\end{proposition}
\begin{proof}
The graph $G$ in Figure \ref{fig-telescope} is obtained by adding vertices $v$ and $w$  to the plane triangulation  $H$:  vertex $v$  connects to all nine vertices of $H$ and vertex $w$ connects to vertices $a$, $b$ and $c$ of $H$.
The graph $H+v$ is maxnil, since it is a cone over a maximal planar graph \cite{Sa}. 
The graph $G$ is the clique sum over $K_3=\big<a, b, c\big>_G$ of  maxnil graphs $H+v$ and $K_4=\big<a,b,c,w\big>_G$. 
The graph $\big<a,b,c,v\big>_{H+v}$ is the only induced $K_4$ subgraph in  $H+v$ containing $a,b$ and $c$ and it is strongly separating in $H+v$.
So, by Lemma \ref{lemmajoin3}, $G$ is maxnil; in particular it has no $K_6$ minor.
The graph $G+vw$ is a clique sum over $K_4=\big<a,b,c,v\big>_G$ of graphs $H+v$ and $K_5=\big<a,b,c,v,w\big>$, both of which are $K_6$ minor free.
Hence, by \cite{J}, $G+vw$ is $K_6$ minor free, so $G$ is not maximal without a $K_6$ minor.
\end{proof}}

\begin{remark} { Starting with the graph $G$ in Proposition \ref{noK6}, one can construct graphs $G_n$ with $n\ge 11$ vertices that are maxnil but not maximal without a $K_6$ minor.
Take $G_{11}=G$ and construct $G_{11+k}$  from $G$ by triangulating the disk bounded by the triangle $efg$ with $k$ new vertices and adding edges between $v$ and these new vertices. 
The argument used in the proof of Proposition \ref{noK6} shows that $G_n$, $n\ge 11$, is maxnil but not maximal without a $K_6$ minor. 
We conjecture $n=11$ is the minimal order of a graph with this property, i.e., every maxnil graph with $n \le 10$ vertices is maximal without a $K_6$ minor.}
 
\end{remark}

\bibliographystyle{amsplain}

\end{document}